\documentclass[reqno,a4paper, 11pt]{amsart}

\usepackage[a4paper=true,pdfpagelabels]{hyperref}
\usepackage{graphicx}
\usepackage[ansinew]{inputenc}
\usepackage{amsfonts,epsfig}
\usepackage{latexsym}
\usepackage{amsmath}
\usepackage{amssymb}
\usepackage{mathabx}
\usepackage{comment}
\usepackage{color}

\newtheorem{theorem}{Theorem}
\newtheorem{lemma}[theorem]{Lemma}
\newtheorem{corollary}[theorem]{Corollary}
\newtheorem{proposition}[theorem]{Proposition}

\newtheorem{lettertheorem}{Theorem}
\newtheorem{letterlemma}[lettertheorem]{Lemma}

\theoremstyle{definition}

\newtheorem{example}[theorem]{Example}

\theoremstyle{remark}

\numberwithin{equation}{section}

\newcommand{\D}{\mathbb{D}}
\newcommand{\DD}{\widehat{\mathcal{D}}}
\newcommand{\DDD}{\widecheck{\mathcal{D}}}

\newcommand{\N}{\mathbb{N}}
\newcommand{\Z}{\mathbb{Z}}

\newcommand{\R}{\mathbb{R}}
\newcommand{\HO}{\mathcal{H}}

\newcommand{\C}{\mathbb{C}}

\newcommand{\zjk}{z_{jk}}

\newcommand{\e}{\varepsilon}

\newcommand{\RR}{\mathcal{D}}
\DeclareMathOperator{\Real}{Re}

\newcommand{\T}{\mathbb{T}}
\newcommand{\Th}{\Theta}

\newcommand{\whw}{\widehat{\omega}}

\newcommand{\CZ}{{E'_\Th}}
\newcommand{\cpt}{{D_\delta}}

\def\a{\alpha}       \def\b{\beta}        
\def\d{\delta}           \def\e{\varepsilon}
     \def\om{\omega}      
\def\s{\sigma}       \def\t{\theta}

\begin{document}

\title[Characterizations for inner functions]{Characterizations for inner functions in certain function spaces}


\thanks{This research was supported in part by Finnish Cultural Foundation and JSPS Postdoctoral Fellowship for North American and European Researchers.}



\author{Atte Reijonen}
\address{Graduate School of Information Sciences, Tohoku University, Aoba-ku, Sendai 980-8579, Japan}
\email{atte.reijonen@uef.fi}

\author{Toshiyuki Sugawa}
\address{Graduate School of Information Sciences, Tohoku University, Aoba-ku, Sendai 980-8579, Japan}
\email{sugawa@math.is.tohoku.ac.jp}

\maketitle



\begin{abstract}
For $\frac12<p<\infty$, $0<q<\infty$ and a certain two-sided doubling weight $\om$, we characterize those inner functions $\Th$ for which
    $$
    \|\Th'\|_{A^{p,q}_\omega}^q=\int_0^1 \left(\int_0^{2\pi} |\Th'(re^{i\t})|^p d\t\right)^{q/p} \om(r)\,dr<\infty.
    $$
Then we show a modified version of this result for $p\ge q$.
Moreover, two additional characterizations for inner functions whose derivative belongs to the Bergman space $A_\om^{p,p}$ are given.
\end{abstract}

\section{Introduction and main results}

Let $\HO(\D)$ be the space of all analytic functions in the open unit disc $\D$ of the complex plane $\C$.
For $0<p<\infty$, the Hardy space $H^p$ consists of those $f\in\HO(\D)$ such that
    \begin{equation*}
    \begin{split}
    M_p(r,f)=\left(\frac{1}{2\pi}\int_{0}^{2\pi} |f(re^{i\t})|^p\,d\t\right)^{1/p}, \quad 0\le r<1,
    \end{split}
    \end{equation*}
is bounded. A function $\om: \D \rightarrow [0,\infty)$ is called a (radial) weight if it is integrable over $\D$ and
$\om(z)=\om(|z|)$ for all $z\in \D$.
For $0<p,q<\infty$ and a weight $\omega$, the weighted mixed norm space $A^{p,q}_\omega$ consists of those $f\in \HO(\D)$ such that
    $$
    \|f\|_{A^{p,q}_\omega}^q=\int_0^1 M_p^q(r,f)\, \om(r)\,dr<\infty.
    $$
For a weight $\om$, we set
    \begin{equation*}
    \widehat{\om}(z)=\widehat{\om}{(|z|)}=\int_{|z|}^1\om(s)\,ds,\quad z\in \D.
    \end{equation*}
A weight $\om$ belongs to $\DD$ if there exists $C=C(\om)\ge 1$ such that $\widehat{\om}(r)\le C\widehat{\om}(\frac{1+r}{2})$ for all $0\le r<1$ \cite{PelRat2015}. Analogously, $\om \in \DDD$ if there exist $K=K(\om)>1$ and $C=C(\om)>1$ such that
    \begin{equation*}\label{Eq:definition-DDD}
    \begin{split}
    \widehat{\om}(r)\ge C \widehat{\om}\left(1-\frac{1-r}{K}\right), \quad 0\le r<1.
    \end{split}
    \end{equation*}
Class $\RR$ of so-called two-sided doubling weights is the intersection of $\DD$ and $\DDD$ \cite{TwoweightII}.
An alternative characterization can be given as follows:
A weight $\om$ belongs to $\RR$ if and only if there exist $C=C(\om)\ge1$, $\a=\a(\om)>0$ and $\b=\b(\om)\ge\a$ such that
    \begin{equation}\label{Eq:definition-R}
    \begin{split}
    C^{-1}\left(\frac{1-r}{1-s}\right)^{\a}\widehat{\om}(s)\le\widehat{\om}(r)\le C\left(\frac{1-r}{1-s}\right)^{\b}\widehat{\om}(s),\quad 0\le r\le s<1.
    \end{split}
    \end{equation}
In many cases this characterization is more practical than the original definition.

An inner function is a bounded analytic function having unimodular radial limits almost everywhere on the boundary $\T=\{z\in\C:|z|=1\}$ \cite{Duren1970, Mashreghi2013}.
An important subclass of inner functions consists of Blaschke products \cite{Colwell1985}.
For a given sequence $\{z_n\}\subset\D$ satisfying $\sum_n (1-|z_n|)<\infty$ and a real constant $\lambda,$ the Blaschke product
with zeros $\{z_n\}$ is defined by
    \begin{equation*}\label{Eq:Blaschke}
    B(z)=e^{i\lambda}\, \prod_{n}\frac{|z_n|}{z_n}\frac{z_n-z}{1-\overline{z}_nz}, \quad z\in \D.
    \end{equation*}
For $z_n=0$, the interpretation $|z_n|/z_n=-1$ is used.
Our first result characterizes those inner functions whose derivative belongs to the mixed norm space $A_\om^{p,q}$ satisfying certain regularity conditions.
This characterization is taking advantage of the fact that the Frostman shift $\Th_a$ of an arbitrary inner function $\Th$, defined by
    $$\Th_a(z)=\frac{\Th(z)-a}{1-\overline{a}\Th(z)}, \quad z\in \D,$$
is a Blaschke product for almost every $a\in \D$. More precisely, Frostman's result states that
the exceptional set $E_\Th$ where $\Th_a$ is not a Blaschke product has logarithmic capacity zero; see for instance \cite[Chapter~2,~Theorem~6.4]{Garnett1981}.

Before our first main result, we define $r_j=1-2^{-j}$ for $j\in \N\cup \{0\}$ and $D_\d=\{z\in \C: |z|\le\delta\}$ for $0<\d<1$.
Moreover, recall that
$f\lesssim g$ if there exists a constant $C=C(\cdot)>0$ such that $f\le Cg$, while $f\gtrsim g$ is understood in an analogous manner.
If $f\lesssim g$ and $f\gtrsim g$, then we write $f\asymp g$.
Here the letter $C=C(\cdot)$ is a positive constant whose value depends only on the parameters indicated in the parentheses,
and may change from one occurrence to another.

\begin{theorem}\label{Theo1b}
Let $\frac12<p<\infty$, $0<q<\infty$, $0<\d<1$ and $\om\in \RR$.  If $\Th$ is an inner function and either
    \begin{itemize}
    \item[\rm(a)] $\frac12<p\le 1$ and $\om$ satisfies
the right-hand inequality of \eqref{Eq:definition-R} for some $\b<2q-\frac{q}{p}$, or
    \item[\rm(b)] $1<p<\infty$, $\om$ satisfies
the right-hand inequality of \eqref{Eq:definition-R} for some $\b<q$ and the left-hand inequality for some $\a>q-\frac{q}{p}$,
    \end{itemize}
then
    \begin{equation}\label{eq:Thm1}
    \|\Th'\|_{A_\om^{p,q}}^q \asymp \sum_n \frac{\whw(r_n)}{(1-r_n)^{q-q/p}}\int_{D_\d}\upsilon_n(a)^{q/p} dA(a),
    \end{equation}
where $\{z_n(a)\}$ is the zero-sequence of $\Th_a$ and $\upsilon_n(a)=\#\{j:r_n\le |z_j(a)|<r_{n+1}\}$.
\end{theorem}

Here and hereafter, $dA(z)$ will stand for the $2$-dimensional Lebesgue measure $dxdy.$
The argument of Theorem~\ref{Theo1b} utilizes the recent results regarding the derivative of inner functions in $A_\om^{p,q}$ \cite{R2017c, R2017b},
and certain estimates for $\Th_a$ originated to \cite{Jevtic2009}.
Using Theorem~\ref{Theo1b} together with a connection between the mixed norm and Besov spaces \cite{Flett1972},
we can prove a streamlined version of \cite[Theorem~3.3]{GGJ2011}; 
see Corollary~\ref{coroBesov} in Section~\ref{Sec4}. Hence it does not come as a surprise that the proofs of Theorem~\ref{Theo1b}
and \cite[Theorem~3.3]{GGJ2011} have some similarities.

Modifying the argument of Theorem~\ref{Theo1b} for $p\ge q$,
we may remove the integral over $D_\d$ in the statement by assuming $a\in \D\setminus E_\Th$; see Theorem~\ref{Theo1} below.
More precisely, we verify as an auxiliary result that the Lebesgue integral
over $D_\d$ can be replaced by a certain integral with respect to a probability
measure supported in a compact subset of $\D$;
see Lemma~\ref{lemma3} in Section~\ref{Sec5}. Using this observation, we can prove the desired result for $p=1$ in a similar manner as in Theorem~\ref{Theo1b}.
Finally the assertion follows from an application of some nesting properties for the derivative of inner functions in the mixed norm spaces.
These nesting properties can be obtained by using a consequence of Theorem~\ref{Theo1b}.

\begin{theorem}\label{Theo1}
Let $\frac12<p<\infty$ and $0<q\le p$. Assume that $\om$, $\Th$ and $\upsilon_n(a)$ are as in Theorem~\ref{Theo1b}.
Then the following statements are equivalent:
\begin{itemize}
    \item[\rm(i)] $\Th'\in A_\om^{p,q}$.
    \item[\rm(ii)] There exists a set $\CZ\subset\D$ of logarithmic capacity zero such that
    \begin{equation}\label{E1:t1}
    \sum_n \frac{\whw(r_n)\upsilon_n(a)^{q/p}}{(1-r_n)^{q-q/p}}<\infty
    \end{equation}
    for every $a\in \D\setminus \CZ.$
    \item[\rm(iii)]  There exists $a\in \D\setminus E_\Th$ such that \eqref{E1:t1} holds.
\end{itemize}
\end{theorem}

For instance, $\om_1(z)=(1-|z|)^{\a}$ and $\om_2(z)=(1-|z|)^\a\left(\log \frac{e}{1-|z|}\right)^\b$ satisfy
the hypotheses of $\om$ in Theorems~\ref{Theo1b}~and~\ref{Theo1} if $\max\{-1,q-\frac{q}p-1\}<\a<\min\{2q-\frac{q}p-1,q-1\}$ and $\b\in \R$.
Using this observation for $\om_1$ together with \cite[Lemma~3.1]{GGJ2011}, it is easy to check that \cite[Theorem~3]{Verbitsky1984} for $q\le p$
follows from Theorem~\ref{Theo1}. Our result contains also the case $\frac12<p<1$, unlike the original theorem.
Moreover, it is worth mentioning that \cite[Theorem~3]{Verbitsky1984} was stated without any proof. By studying $A_{\a}^{p,q}=A_{\om_1}^{p,q}$, we can also show that
hypothesis (b) in Theorems~\ref{Theo1b}~and~\ref{Theo1} is natural in a certain sense. More precisely, the only inner functions whose derivative belongs to
$A_{\a}^{p,q}$ for $\a\le q-\frac{q}p-1$ are finite Blaschke products \cite{R2017b}; and if $\a>q-1$, then the derivative of every inner function belongs to $A_{\a}^{p,q}$ by the Schwarz-pick lemma.

Next we turn our attention to the case of the weighted Bergman space $A_\om^p$, which is the mixed norm space with $p=q$.
If $\om(z)=(1-|z|)^\a$ for some $-1<\a<\infty$, then the notation $A_\a^p$ is used for $A_\om^p.$
Our first characterizations for inner functions $\Th$ whose derivative belongs to $A_\om^p$ are straightforward consequences of Theorem~\ref{Theo1}.
In addition, we give a generalization of the equivalence $\rm (a) \Leftrightarrow (b)$ in \cite[Theorem~1]{Grohn2016}.
The proof of this result is based on the existence of approximating Blaschke products \cite{Cohn1986}
and some estimates for $\|\Th'\|_{A_\om^p}$ \cite{PR2016, PGRR2016}.
The argument used here is essentially different from that used in \cite{Grohn2016}.
Applying the above-mentioned tools, we can also prove a
characterization which utilizes the so-called Carleson curve $\Gamma_\e=\Gamma_\e(\Th)$
(in the sense of W.~S.~Cohn) associated with $0<\e<1$ and an inner function $\Th$.
This characterization together with our auxiliary results generalizes \cite[Theorem~7]{Pelaez2008}.
The construction of Carleson curves in the case of the upper half-plane can be found in \cite[pp.~328--330]{Garnett1981}, and
its unit disc analogue has been studied, for instance, in \cite{Cohn1983, Cohn1986}.
Some properties of $\Gamma_\e$ are recalled also in Section~\ref{Sec6}.

\begin{theorem}\label{thm3}
Let $\frac12<p<\infty$, $\om\in \RR$, $\Th$ be an inner function and $\{z_n(a)\}$ the zero-sequence of $\Th_a$. Moreover, assume either
\begin{itemize}
    \item[\rm(a)] $\frac12<p\le 1$ and $\om$ satisfies
the right-hand inequality of \eqref{Eq:definition-R} for some $\b<2p-1$, or
    \item[\rm(b)] $1<p<\infty$, $\om$ satisfies
the right-hand inequality of \eqref{Eq:definition-R} for some $\b<p$ and the left-hand inequality for some $\a>p-1$.
    \end{itemize}
Then the following statements are equivalent:
\begin{itemize}
    \item[\rm(i)] $\Th'\in A_\om^{p}$.
    \item[\rm(ii)] There exists a set  $\CZ \subset\D$ of logarithmic capacity zero such that
    \begin{equation}\label{C1:t1}
    \sum_n \frac{\whw(z_n(a))}{(1-|z_n(a)|)^{p-1}}<\infty
    \end{equation}
    for every $a\in \D\setminus \CZ$.
    \item[\rm(iii)] There exists $a\in \D\setminus E_\Th$ such that \eqref{C1:t1} holds.
    \item[\rm(iv)] There exists $0<C<1$ such that
      \begin{equation*}
    \int_{\{z\in \D: |\Th(z)|<C\}} \frac{\whw(z)}{(1-|z|)^{p+1}}\,dA(z)<\infty.
    \end{equation*}
    \item[\rm(v)] There exists $0<\e<1$ such that
      \begin{equation*}
    \int_{\Gamma_\e} \frac{\whw(z)}{(1-|z|)^{p}}\,|dz|<\infty.
    \end{equation*}
\end{itemize}
\end{theorem}

As mentioned above, Theorem~\ref{thm3} implies a part of \cite[Theorem~1]{Grohn2016}.
In addition, the essential contents of classical results
\cite[Theorem~6.2]{Ahern1979} and \cite[Theorem~3]{Cohn1983} are consequences of Theorem~\ref{thm3}.
All of these results are contained in the following corollary.

\begin{corollary}\label{coro:Hp}
Let $\frac12<p<1$, $\Th$ be an inner function and $\{z_n(a)\}$ the zero-sequence of $\Th_a$. Then the following statements are equivalent:
\begin{itemize}
    \item[\rm(i)] $\Th'\in H^p$.
    \item[\rm(ii)] $\Th'\in A_\a^{p+\a+1}$ for every $-1<\a<\infty$.
    \item[\rm(iii)] $\Th'\in A_\a^{p+\a+1}$ for some $-1<\a<\infty$.
    \item[\rm(iv)] There exists a set $\CZ\subset \D$ of logarithmic capacity zero such that
    \begin{equation}\label{Eq:Hp-cond}
    \begin{split}
    \sum_n (1-|z_n(a)|)^{1-p}<\infty
    \end{split}
    \end{equation}
    for every $a\in \D\setminus \CZ$.
    \item[\rm(v)] There exists $a\in \D\setminus E_\Th$ such that \eqref{Eq:Hp-cond} holds.
    \item[\rm(vi)] There exists $0<C<1$ such that
      \begin{equation*}
    \int_{\{z\in \D: |\Th(z)|<C\}} \frac{dA(z)}{(1-|z|)^{p+1}}<\infty.
    \end{equation*}
    \item[\rm(vii)] There exists $0<\e<1$ such that
      \begin{equation*}
    \int_{\Gamma_\e} \frac{|dz|}{(1-|z|)^{p}}<\infty.
    \end{equation*}
\end{itemize}
\end{corollary}

The remainder of this note is organized as follows. Some auxiliary results are stated in Section~\ref{Sec2}.
Theorems~\ref{Theo1b},~\ref{Theo1}~and~\ref{thm3} are proved in Sections~\ref{Sec3},~\ref{Sec5}~and~\ref{Sec6}, respectively.
Consequences of Theorem~\ref{Theo1b} are stated in Section~\ref{Sec4}, and the proof of Corollary~\ref{coro:Hp} can be found in Section~\ref{Sec7}.
In addition, the last section contains an example and some remarks.

\section{Auxiliary results}\label{Sec2}

We begin by stating a sufficient condition for the derivative of a Blaschke product $B$ to be in $A_\om^{p,q}$ \cite{R2017c}. In addition, it is mentioned that the condition is necessary if the zero-sequence of $B$ is a
finite union of separated sequences. Before this result we recall that a sequence $\{z_n\}\subset \D$ is called separated if there exists $\d=\d(\{z_n\})>0$ such that
    \begin{equation*}\label{separated}
    \left|\frac{z_n-z_k}{1-\overline{z}_nz_k}\right|>\d, \quad n\neq k.
    \end{equation*}

\begin{letterlemma}\label{lemma1}
Let $\frac12<p<\infty$, $0<q<\infty$, $\om\in \RR$ and $B$ be the Blaschke product with zeros $\{z_j\}$. If either
    \begin{itemize}
    \item[\rm(a)] $\frac12<p\le 1$ and $\om$ satisfies
the right-hand inequality of \eqref{Eq:definition-R} for some $\b<2q-\frac{q}{p}$, or
    \item[\rm(b)] $1<p<\infty$, $\om$ satisfies
the right-hand inequality of \eqref{Eq:definition-R} for some $\b<q$ and the left-hand inequality for some $\a>q-\frac{q}{p}$,
    \end{itemize}
then
    \begin{equation*}
    \|B'\|_{A_\om^{p,q}}^q \lesssim \sum_n \frac{\whw(r_n)\upsilon_n^{q/p}}{(1-r_n)^{q-q/p}},
    \end{equation*}
where $\upsilon_n=\#\{j:r_n\le |z_j|<r_{n+1}\}$.
If, in addition to \emph{(a)} or \emph{(b)}, $\{z_j\}$ is a finite union of separated sequences, then
     \begin{equation*}
    \|B'\|_{A_\om^{p,q}}^q \asymp \sum_n \frac{\whw(r_n)\upsilon_n^{q/p}}{(1-r_n)^{q-q/p}}.
    \end{equation*}
\end{letterlemma}

We say that a weight $\om$ belongs to $\DD_p$ for $0<p<\infty$ if
    \begin{equation*} \label{Eq:DDp}
    \sup_{0<r<1}\frac{(1-r)^p}{\widehat{\om}(r)}\int_0^r\frac{\om(s)}{(1-s)^p}\,ds<\infty,
    \end{equation*}
and $\om \in\DDD_p$ if
    \begin{equation*} \label{Eq:DDp}
    \sup_{0<r<1}\frac{(1-r)^p}{\widehat{\om}(r)}\int_r^1\frac{\om(s)}{(1-s)^p}\,ds<\infty.
    \end{equation*}
It is worth noting that (a) and (b) in Lemma~\ref{lemma1} for $p\ge q$ can be replaced by the following conditions respectively:
    \begin{itemize}
    \item[\rm(A)] $\frac12<p\le 1$ and $\om\in \DD_{2q-q/p}$,
    \item[\rm(B)] $1<p<\infty$ and $\om\in \DD_q \cap \DDD_{q-q/p}$.
    \end{itemize}
This observation is relevant because conditions (a) and (b) imply (A) and (B), respectively.
More precisely, if the right-hand inequality of \eqref{Eq:definition-R} is satisfied for some $\b=\b(\om)<p$, then $\om\in \DD_p$.
Similarly, if the left-hand inequality is satisfied for some $\a=\a(\om)>p$, then $\om\in \DDD_p$.
The validity of these implications can be checked by straightforward calculations based on integration by parts; see \cite{R2017c} for details.
In addition, we recall that $\om \in \DD$ if and only if $\om \in \DD_p$ for some $p$ \cite{PelRat2015}.

The next auxiliary result shows that, for $\om \in \RR \cap \DD_q$ and an inner function $\Th$,
we may use the Schwarz-Pick lemma inside the norm $\|\Th'\|_{A^{p,q}_\omega}$
without losing any essential information \cite{R2017b}. In addition, we give some modified asymptotic estimates for $\|\Th'\|_{A^{p,q}_\omega}$,
which are consequences of the following fact \cite{R2017b}: For $0<p,q<\infty $ and $\om\in \RR$,
   \begin{equation*}
    \begin{split}
    \|f\|_{A_\om^{p,q}}^q \asymp \int_0^1 M_p^q(r,f)\,\frac{\whw(r)}{1-r}\,dr
    \end{split}
    \end{equation*}
for all $f\in \HO(\D)$.

\begin{letterlemma}\label{lemma2}
Let $0<p,q<\infty$, $\om \in \RR \cap \DD_q$ and $\Th$ be an inner function. Then
    \begin{equation*}
    \begin{split}
    \|\Th'\|_{A^{p,q}_\omega}^q &\asymp \int_0^1 \left(\int_0^{2\pi} \left(\frac{1-|\Th(re^{i\t})|}{1-r}\right)^p d\t\right)^{q/p} \om(r)\, dr \\
    &\asymp \int_0^1 \left(\int_0^{2\pi} \left(\frac{1-|\Th(re^{i\t})|}{1-r}\right)^p d\t\right)^{q/p} \frac{\whw(r)}{1-r}\, dr \\
    &\asymp \int_0^1 \left(\int_0^{2\pi} |\Th'(re^{i\t})|^p d\t\right)^{q/p} \frac{\whw(r)}{1-r}\, dr.
    \end{split}
    \end{equation*}
\end{letterlemma}

We close this section by recalling that the counterparts of Lemmas~\ref{lemma1}~and~\ref{lemma2}
for $A_\om^p$ were originally proved in \cite{PR2016,PGRR2016}. In addition, we note that Lemma~\ref{lemma2}
and the first part of Lemma~\ref{lemma1} for $p=q$ are valid also if the hypothesis $\om\in \RR$ is replaced by $\om\in \DD$.

\section{Proof of Theorem~\ref{Theo1b}}\label{Sec3}

Let us begin by proving a modification of \cite[Lemma~4.6]{Jevtic2009}.

\begin{lemma}\label{Lemma-Jevtic}
Let $0<p<\infty$ and $0<\d<1$. Then there exists $C=C(p,\d)>0$ such that
    \begin{equation*}
    \int_\cpt \left(\log\left|\frac{1-\overline{a}z}{z-a}\right|\right)^p dA(a)\le C(1-|z|)^p, \quad z\in \D.
    \end{equation*}
\end{lemma}

\begin{proof}
If $|z|<(1+\d)/2$, then the assertion follows by observing that there exists a constant $M=M(p)>0$ such that
    $$\int_\cpt \left(\log\left|\frac{1-\overline{a}z}{z-a}\right|\right)^p dA(a)<M<\infty.$$
Hence we may assume $|z|\ge (1+\d)/2$. Since
    $$\left|\frac{1-\overline{a}z}{z-a}\right|^2=\frac{(1-|z|^2)(1-|a|^2)}{|z-a|^2}+1,$$
we have
    $$\log\left|\frac{1-\overline{a}z}{z-a}\right|^2\le \log\left(8(1-\d)^{-2}(1-|z|)+1\right)
    \le 8(1-\d)^{-2}(1-|z|), \quad a\in \cpt.$$
Consequently, the assertion follows.
\end{proof}

For $x\in \R$ and a weight $\om$, we set $\om_x(z)=\om(z)(1-|z|)^x$ for all $z\in\D$.
If  $0<x<\infty$ and $\om \in \RR$, then
    \begin{equation}\label{Eq:thm1-5}
   \begin{split}
    \widehat{\om_x}(z) \asymp \whw(z)(1-|z|)^x, \quad z\in \D;
   \end{split}
   \end{equation}
see the proof of \cite[Corollary 7]{PGRR2016}.
It follows that
     \begin{equation}\label{Eq:definition-R-wx}
    \begin{split}
    \left(\frac{1-r}{1-s}\right)^{\a+x}\widehat{\om_x}(s)\lesssim\widehat{\om_x}(r)\lesssim \left(\frac{1-r}{1-s}\right)^{\b+x}\widehat{\om_x}(s),\quad 0\le r\le s<1,
    \end{split}
    \end{equation}
where $\a$ and $\b$ are from \eqref{Eq:definition-R}.
With these preparations we are ready to prove Theorem~\ref{Theo1b}.

\medskip

\noindent
\emph{Proof of Theorem~\ref{Theo1b}.}
Since
    \begin{equation}\label{Eq:thm1-0}
    \Th_a'(z)=\Th'(z)\frac{1-|a|^2}{(1-\overline{a}\Th(z))^2},
    \end{equation}
we obtain $|\Th'(z)|\asymp |\Th_a'(z)|$ for $z\in \D$ and $a\in D_\d$. Hence Lemma~\ref{lemma1} yields
    $$\|\Th'\|_{A_\om^{p,q}}^q \asymp \int_{D_\d}\|\Th_a'\|_{A_\om^{p,q}}^q \,dA(a)\lesssim
    \sum_n \frac{\whw(r_n)}{(1-r_n)^{q-q/p}}\int_{D_\d}\upsilon_n(a)^{q/p} dA(a);$$
and consequently, the upper bound for $\|\Th'\|_{A_\om^{p,q}}$ is proved.

Let $1\le p<\infty$ and $0<q<\infty$.
Since
    $$\log \frac{1}{|\Th_a(z)|}\ge \sum_n \frac{(1-|z|^2)(1-|z_n(a)|^2)}{|1-\overline{z}_n(a)z|^2}, \quad z\in \D,$$
by \cite[Chapter~7, Lemma~1.2]{Garnett1981}, the super-additivity of $x^p$ for $0<x\le 1$ and the Forelli-Rudin estimate \cite[Theorem~1.7]{Hedenmalm2000} give
    \begin{equation}\label{Eq:thm1-1}
    \begin{split}
    \upsilon_n(a)(1-r_n) \lesssim \int_0^{2\pi} \left(\log\frac{1}{|\Th_a(re^{i\t})|}\right)^p d\t, \quad r_n\le r<r_{n+1},
    \end{split}
    \end{equation}
as observed in \cite[Corollary~4.5]{Jevtic2009}.
Using \eqref{Eq:thm1-1} together with the hypothesis $\om\in \DD$, we obtain
    \begin{equation}\label{Eq:thm1-3}
    \begin{split}
    S&:=\sum_n \frac{\whw(r_n)}{(1-r_n)^{q-q/p}}\int_{D_\d}\upsilon_n(a)^{q/p} dA(a) \\
    &\asymp \sum_n \int_{r_n}^{r_{n+1}} \frac{\whw(r)}{(1-r)^{q+1}}\int_{D_\d}\upsilon_n(a)^{q/p}(1-r_n)^{q/p} dA(a)\,dr \\
    &\lesssim \sum_n \int_{r_n}^{r_{n+1}} \frac{\whw(r)}{(1-r)^{q+1}}
    \int_{D_\d}\left(\int_0^{2\pi} \left(\log \frac{1}{|\Th_a(re^{i\t})|}\right)^p d\t\right)^{q/p} dA(a)\,dr.
    \end{split}
    \end{equation}

If $p<q$, then \eqref{Eq:thm1-3}, Minkowski's inequality \cite[Theorem 202]{HLP:ineq}  and Lemma~\ref{Lemma-Jevtic} for $z=\Th(re^{i\t})$ yield
    \begin{equation*}
    \begin{split}
    S&\lesssim \sum_n \int_{r_n}^{r_{n+1}} \frac{\whw(r)}{(1-r)^{q+1}}
     \left(\int_0^{2\pi}\left(\int_{D_\d}\left(\log \frac{1}{|\Th_a(re^{i\t})|}\right)^q  dA(a) \right)^{p/q} d\t\right)^{q/p}dr \\
    &\lesssim \int_0^1 \left(\int_0^{2\pi} \left(\frac{1-|\Th(re^{i\t})|}{1-r}\right)^p\,d\theta\right)^{q/p} \frac{\whw(r)}{1-r}\,dr.
    \end{split}
    \end{equation*}
For $p\ge q$, we use \eqref{Eq:thm1-3}, H\"older's inequality and Lemma~\ref{Lemma-Jevtic} to obtain
    \begin{equation*}
    \begin{split}
    S&\lesssim \sum_n \int_{r_n}^{r_{n+1}} \frac{\whw(r)}{(1-r)^{q+1}}
    \left(\int_0^{2\pi}\int_{D_\d} \left(\log \frac{1}{|\Th_a(re^{i\t})|}\right)^p dA(a)\, d\t  \right)^{q/p} dr \\
    &\lesssim \int_0^1 \left(\int_0^{2\pi} \left(\frac{1-|\Th(re^{i\t})|}{1-r}\right)^p\,d\t\right)^{q/p} \frac{\whw(r)}{1-r}\,dr.
    \end{split}
    \end{equation*}
Finally the lower bound of $\|\Th'\|_{A_\om^{p,q}}$ for $1\le p<\infty$ follows from these inequalities and Lemma~\ref{lemma2}.
Thus we have shown \eqref{eq:Thm1} when $p\ge 1$.

Let $\frac12<p<1$ and $0<q<\infty$.
Put $x=q/p-q$ and assume, by the hypotheses $\om \in \RR$ and (a), $\a+x>0$ and $\b+x<q/p$.
Finally asymptotic equation \eqref{Eq:thm1-5}, \eqref{eq:Thm1} with $p$ and $q$ being replaced by
$1$ and $q/p,$ respectively, and the Schwarz-Pick lemma yield
    \begin{equation*}
    \begin{split}
    S\asymp \sum_n \widehat{\om_x}(r_n) \int_{D_\d}\upsilon_n(a)^{q/p} dA(a)
    \asymp \|\Th'\|_{A_{\om_x}^{1,q/p}}^{q/p}
    \le  \|\Th'\|_{A_\om^{p,q}}^q.
    \end{split}
    \end{equation*}
This completes the proof.  \hfill$\Box$

\medskip

Note that the proof of the lower bound
    \begin{equation*}
    \|\Th'\|_{A_\om^{p,q}}^q \gtrsim \sum_n \frac{\whw(r_n)}{(1-r_n)^{q-q/p}}\int_{D_\d}\upsilon_n(a)^{q/p} dA(a)
    \end{equation*}
relies on Lemma~\ref{lemma2}, not Lemma~\ref{lemma1}. In particular, this means that for the lower bound it suffices to assume only the hypotheses of Lemma~\ref{lemma2}.

\section{Consequences of Theorem~\ref{Theo1b}}\label{Sec4}

The first consequence of Theorem~\ref{Theo1b} asserts that
the derivative of an inner function $\Th$ belongs to $A_\om^{p,q}$ if and only if $\Th'\in A_{\om_x}^{p+xp/q, q+x}$ for every/some $0<x<\infty$.
Note that this result was originally proved in \cite{R2017c}.
The argument there relies on the existence of approximating Blaschke products \cite{Cohn1986}, unlike the proof here.

\begin{corollary}\label{coro2}
Let $\frac12<p<\infty$ and $0<q,x<\infty$. Assume that $\om$ and $\Th$ are as in Theorem~\ref{Theo1b}.
Then
    $$\|\Th'\|_{A_\om^{p,q}}^q\asymp \|\Th'\|_{A_{\om_x}^{p+xp/q,q+x}}^{q+x}.$$
\end{corollary}

\begin{proof}
Let $p'=p+xp/q$ and $q'=q+x$. Then, by \eqref{Eq:thm1-5}, we have
    $$\sum_n \frac{\whw(r_n)}{(1-r_n)^{q-q/p}}\int_{D_\d}\upsilon_n(a)^{q'/p'}dA(a)
    \asymp \sum_n \frac{\widehat{\om_x}(r_n)}{(1-r_n)^{q'-q'/p'}}\int_{D_\d}\upsilon_n(a)^{q'/p'}dA(a),$$
where $\upsilon_n(a)$ is as in Theorem~\ref{Theo1b}.
Hence the assertion follows from Theorem~\ref{Theo1b} by showing that one of the following conditions holds:
    \begin{itemize}
    \item[\rm(i)] $\frac12<p'\le 1$ and $\om_x$ satisfies
    the right-hand inequality of \eqref{Eq:definition-R-wx} for some $\b+x<2q'-\frac{q'}{p'}$
    and the left-hand inequality for some $\a+x>0$,
    \item[\rm(ii)] $1<p'<\infty$, $\om_x$ satisfies
    the right-hand inequality of \eqref{Eq:definition-R-wx} for some $\b+x<q'$ and the left-hand inequality for some $\a+x>q'-\frac{q'}{p'}$.
    \end{itemize}
The validity of (i) or (ii) can be checked by considering the cases $\frac12<p'\le 1$, $1<p'\le 1+xp/q$ and $p'>1+xp/q$ separately;
see \cite{R2017c} for details. Consequently, the proof is complete.
\end{proof}

Next we turn our attention to the Besov space.
For $0<\a<\infty$ and an analytic function $f(z)=\sum_n a_n z^n$, the fractional derivative of order $\a$ is defined by
    $$D^{\a} f(z)=\sum_n (n+1)^\a a_n z^n, \quad z\in \D.$$
Note that, for  $f\in \HO(\D)$, $n\in \N$ and $0<p<\infty$, we have $M_p(r,f^{(n)}) \asymp M_p(r,D^n f)$
with comparison constants independent of $r$ \cite{Flett1972}.
For $0<p,q<\infty$ and $0\le \a<\infty$, the Besov space $B_\a^{p,q}$ consists of those $f\in \HO(\D)$ such that
    $$\|f\|_{B_\a^{p,q}}^q=\int_0^1 M_p^q(r,D^{1+\a}f)(1-r)^{q-1}\,dr<\infty.$$

\begin{corollary}\label{coroBesov}
Let $\Th$ be an inner function, $0<\d<1$ and $0<p,q,\a<\infty$ be such that $\max\{0,\frac{1}{p}-1\}<\a<\frac1p$.
Then $\Th\in B_\a^{p,q}$ if and only if
    \begin{equation}\label{Eq:lemmaC}
    \sum_n (1-r_n)^{q/p-\a q}\int_{D_\d}\upsilon_n(a)^{q/p} dA(a)<\infty,
    \end{equation}
where $\upsilon_n(a)$ as in Theorem~\ref{Theo1b}.
\end{corollary}

Set $K_\d=\{z\in \C: \d\le|z|\le 1-\d\}$ for $0<\d<\frac12$, and
recall that \cite[Theorem~3.3]{GGJ2011} is a corresponding result where \eqref{Eq:lemmaC} is replaced by the condition
    $$\int_{K_\d}\left(\sum_n(1-r_n)^{q/p-\a q}\upsilon_n(a)^{q/p} \right)^{p/q}dA(a)<\infty.$$
One could say that Corollary~\ref{coroBesov} is a streamlined version of \cite[Theorem~3.3]{GGJ2011}, or
a generalization of the main result of \cite{Jevtic1990}.
Before the proof we underline that our argument for $\a\ge 1$ takes advantage of the original result.

\medskip

\noindent
\emph{Proof of Corollary~\ref{coroBesov}.}
Let $0<\a<1$. Then Theorem~\ref{Theo1b} together with \cite[Theorem~6]{Flett1972} yields
    \begin{equation*}\label{Eq:Flett}
    \begin{split}
    \|\Th\|_{B_\a^{p,q}}^q&\asymp \int_0^1 M_p^q(r,D^1 \Th)(1-r)^{(1-\a)q-1}dr
    \asymp \|\Th'\|_{A_{(1-\a)q-1}^{p,q}}^q \\
    &\asymp \sum_n (1-r_n)^{q/p-\a q}\int_{D_\d}\upsilon_n(a)^{q/p} dA(a).
    \end{split}
    \end{equation*}
Note that for the last asymptotic equation it suffices to check that $\om(z)=(1-|z|)^{(1-\a)q-1}$ satisfies the hypotheses of Theorem~\ref{Theo1b}.
This gives the assertion for $0<\a<1$.

Let $1\le\a<\infty$ and $\a<t<\infty$. By \cite[Lemma~3.4]{GGJ2011}, we know $B_{\a}^{p,q}\subset B_{\a/t}^{pt,qt}$.
Moreover, \cite[Corollary~3.6]{GGJ2011} for $t'=1/t$ implies $\Th\in B_{\a}^{p,q}$ if $\Th \in B_{\a/t}^{pt,qt}$ for $\a>\frac{1}{p}-1$.
Applying these facts together with the previous case, it is easy to verify the assertion for $\a\ge 1$. This completes the proof.
\hfill$\Box$

\section{Proof of Theorem~\ref{Theo1}}\label{Sec5}

Let us begin by stating an auxiliary result, which can be proved in a similar manner as Lemma~\ref{Lemma-Jevtic}.

\begin{lemma}\label{lemma3}
Let $0<\d<1$ and $\s$ be a probability measure supported in
$\cpt$ and satisfying
    \begin{equation}\label{Eq:l3_0}
    \sup_{z\in \D} \int_\cpt \log\left|\frac{1-\overline{a}z}{z-a}\right|\,d\s(a)=M<\infty.
    \end{equation}
Then there exists $C=C(\d,M)>0$ such that
    \begin{equation*}
    \int_\cpt \log\left|\frac{1-\overline{a}z}{z-a}\right|\,d\s(a)\le C(1-|z|), \quad z\in \D.
    \end{equation*}
\end{lemma}

Before the proof of Theorem~\ref{Theo1}, we recall that a compact set $K\subset \D$ has a positive logarithmic (inner) capacity if there exits
a non-zero probability measure $\s$ supported in $K$ and satisfying \eqref{Eq:l3_0}.
For details, see Section 12 as well as Section 2 in \cite[Chapter~III]{Tsuji1959}.

\medskip

\noindent
\emph{Proof of Theorem~\ref{Theo1}.}
If \eqref{E1:t1} holds for some $a\in \D \setminus E_\Th$, then $\Th'\in A_\om^{p,q}$ by Lemma~\ref{lemma1} and \eqref{Eq:thm1-0}.
Consequently, condition (iii) implies (i). Moreover, since $(\D\setminus \CZ)\cap (\D\setminus E_\Th) \neq \emptyset$,
the implication $\rm(ii)\Rightarrow (iii)$ is clear. Hence it suffices to show $\rm(i)\Rightarrow (ii)$.

Assume $\Th'\in A_\om^{p,q}$ for some $p\ge 1$,
let $0<\d<1$ and $\s$ be a probability measure supported in
$\cpt$ and satisfying \eqref{Eq:l3_0}.
For condition (ii) it suffices to prove
    \begin{equation}\label{Eq:proof-thm2_1}
    I:=\int_\cpt \sum_n \frac{\whw(r_n)\upsilon_n(a)^{q/p}}{(1-r_n)^{q-q/p}} \,d\s(a)<\infty
    \end{equation}
because then
    $$\s\left(\left\{a\in \cpt:\sum_n \frac{\whw(r_n)\upsilon_n(a)^{q/p}}{(1-r_n)^{q-q/p}}=\infty  \right\}\right)=0.$$
Using \eqref{Eq:thm1-1} for $p'=1$, H\"older's inequality together with the hypothesis $q\le p$ and Lemma~\ref{lemma3} for $z=\Th(re^{i\t})$, we obtain
    \begin{equation}\label{Eq:proof-thm2}
    \begin{split}
    I&\asymp \sum_n \int_{r_n}^{r_{n+1}} \frac{\whw(r)}{(1-r)^{q+1}}\int_{\cpt}\upsilon_n(a)^{q/p}(1-r_n)^{q/p} d\s(a)\,dr \\
    &\lesssim \sum_n \int_{r_n}^{r_{n+1}} \frac{\whw(r)}{(1-r)^{q+1}}
    \int_{\cpt}\left(\int_0^{2\pi} \log \left(\frac{1}{|\Th_a(re^{i\t})|}\right) d\t\right)^{q/p} d\s(a)\,dr \\
    &\le\int_0^1 \frac{\whw(r)}{(1-r)^{q+1}}
    \left(\int_0^{2\pi}\int_{\cpt} \log \left(\frac{1}{|\Th_a(re^{i\t})|}\right)d\s(a)\, d\t\right)^{q/p} dr \\
    &\lesssim \int_0^1
    \left(\int_0^{2\pi} \frac{1-|\Th(re^{i\t})|}{1-r} \,d\t\right)^{q/p}\frac{\whw(r)}{(1-r)^{q+1-q/p}} \,dr \\
    &\le \int_0^1
    \left(\int_0^{2\pi} \frac{1-|\Th(re^{i\t})|}{1-r} \,d\t\right)^{q/p}\frac{\widehat{\om_{x}}(r)}{1-r} \,dr,
    \end{split}
    \end{equation}
where $x=q/p-q$.

Next we verify some properties for $\om_x$. By the second part of hypothesis (b) and its consequence $\om\in \DDD_{-x}$, we find $\a=\a(\om)>-x$ such that
    $$\frac{\widehat{\om_x}(s)}{(1-s)^{\a+x}}
    \lesssim \frac{\widehat{\om}(s)}{(1-s)^{\a}}
    \lesssim \frac{\widehat{\om}(r)}{(1-r)^{\a}}
    \le \frac{\widehat{\om_x}(r)}{(1-r)^{\a+x}}, \quad 0\le r\le s<1,$$
and
    $$\whw(t)(1-t)^x\lesssim 2^\a\whw\left(\frac12\right)(1-t)^{\a+x} \longrightarrow 0^+, \quad t \rightarrow 1^-.$$
Consequently, an integration by parts together with hypothesis (b) gives
    \begin{equation*}
    \begin{split}
    \left(\frac{1-r}{1-s}\right)^{\a+x}\widehat{\om_x}(s)&\lesssim \widehat{\om_x}(r)=\whw(r)(1-r)^x-x\int_r^1\frac{\whw(t)}{(1-t)^\a}(1-t)^{\a+x-1}dt \\
    &\lesssim \whw(r)(1-r)^x \lesssim \left(\frac{1-r}{1-s}\right)^{\b+x}\widehat{\om_x}(s), \quad 0\le r\le s<1,
    \end{split}
    \end{equation*}
for some $\a=\a(\om)>-x$ and $\b=\b(\om)<q$.

Finally \eqref{Eq:proof-thm2}, Lemma~\ref{lemma2} and  Corollary~\ref{coro2} for $p'=1$, $q'=q/p$ and $x'=-x$ yield
    $$I\lesssim \|\Th'\|_{A_{\om_x}^{1,q/p}}^{q/p}\asymp \|\Th'\|_{A_{\om}^{p,q}}^{q}< \infty.$$
Hence estimate \eqref{Eq:proof-thm2_1} is satisfied for $p\ge 1$. Since the case $\frac12<p<1$ can be verified
by imitating the end part of the proof of Theorem~\ref{Theo1b}, condition (i) implies (ii).
This completes the proof.  \hfill$\Box$

\medskip

It is worth noting that we can slightly weaken the hypotheses for $\om$ in Theorem~\ref{Theo1}:
Condition (a) can be replaced by the hypothesis $\om\in \DD_{2q-q/p}$, and the first part of (b) by $\om\in \DD_{q}$.
This is due to the alternative version of Lemma~\ref{lemma1} for $q\le p$, mentioned in Section~\ref{Sec2}.
More precisely, we have to first prove a modification of Corollary~\ref{coro2} for $q\le p$, and then apply this result in the proof.

\section{Proof of Theorem~\ref{thm3}}\label{Sec6}

Recall that a sequence $\{z_n\}\subset \D$ is said to be uniformly separated if
    $$
    \inf_{n\in\N}\prod_{k\ne n}\left|\frac{z_k-z_n}{1-\overline{z}_kz_n}\right|>0.
    $$
By \cite[Theorem~2.1]{Cohn1986}, for every inner function $\Th$, there exists a Blaschke product $B_\Th$ with uniformly separated zeros $\{z_n\}$ such that
$1-|\Th(z)|\asymp 1-|B_\Th(z)|$ for all $z\in \D$. $B_\Th$ is called an approximating Blaschke product of $\Th$.
Using the existence of approximating Blaschke products together with our auxiliary results, we prove the following
proposition which implies the equivalence $\rm (i)\Leftrightarrow (iv)$ in Theorem~\ref{thm3}.

\begin{proposition}\label{prop1}
Let $\frac12<p<\infty$, $\om\in \RR$ and $\Th$ be an inner function. Moreover, assume either $\frac12<p\le 1$ and $\om\in \DD_{2p-1}$, or
$1<p<\infty$ and $\om\in \DD_p \cap \DDD_{p-1}$. Then there exists $C=C(\Th)\in (0,1)$ such that
    \begin{equation*}
    \|\Th'\|_{A_\om^p}^p\asymp I_C:=\int_{\{z\in \D: |\Th(z)|<C\}} \frac{\whw(z)}{(1-|z|)^{p+1}}\,dA(z),
    \end{equation*}
where the comparison constants may depend on $p$, $\om$, $\Th$ and $C$.
\end{proposition}

\begin{proof}
For any $0<C<1$, Lemma~\ref{lemma2} yields
    \begin{equation*}
    \begin{split}
    \|\Th'\|_{A_\om^p}^p&\asymp \int_{\D} \left(\frac{1-|\Th(z)|}{1-|z|}\right)^p \frac{\whw(z)}{1-|z|}\,dA(z)\ge (1-C)^pI_C.
    \end{split}
    \end{equation*}
Hence the lower bound for $\|\Th'\|_{A_\om^p}$ is proved.

Let $B_\Th$ be an approximating Blaschke product of $\Th$ with zeros $\{z_n\}$.
Since $\{z_n\}$ is (uniformly) separated, we find $0<\d<1$ such that
discs $\Delta(z_n)=\{z:|z_n-z|<\d(1-|z_n|)\}$ are pairwise disjoint.
Hence, using Lemma~\ref{lemma2} and \cite[Theorem~1]{PGRR2016}
together with the hypotheses for $\om$, we obtain
    \begin{equation*}
    \begin{split}
    \|\Th'\|_{A_\om^p}^p&\asymp \int_{\D} \left(\frac{1-|\Th(z)|}{1-|z|}\right)^p \om(z)\,dA(z)
    \asymp \int_{\D} \left(\frac{1-|B_\Th(z)|}{1-|z|}\right)^p \om(z)\,dA(z) \\
    &\asymp \|B_\Th'\|_{A_\om^p}^p \asymp \sum_n \frac{\whw(z_n)}{(1-|z_n|)^{p-1}}
    \asymp \sum_n \int_{\Delta(z_n)}\,dA(z) \frac{\whw(z_n)}{(1-|z_n|)^{p+1}}\\
    &\asymp \sum_n \int_{\Delta(z_n)} \frac{\whw(z)}{(1-|z|)^{p+1}}\,dA(z)
    = \int_{\bigcup_n \Delta(z_n)} \frac{\whw(z)}{(1-|z|)^{p+1}}\,dA(z).
    \end{split}
    \end{equation*}
Consequently, it suffices to find constants $C$ and $D$ such that $0<C,D<1$ and
    $$\bigcup_n \Delta(z_n)\subset \{z\in \D: |B_\Th(z)|<D\} \subset \{z\in \D: |\Th(z)|<C\}.$$
Since
    $$|B_\Th(z)|\le \frac{|z_n-z|}{|1-\overline{z}_n z|}\le \frac{|z_n-z|}{1-|z_n|}<\d, \quad  z\in \Delta(z_n),$$
the first inclusion is valid for $D=\d$. If $|B_\Th(z)|<D$ for some $0<D<1$, then we find $M=M(\Th)<\frac{1}{1-D}$ such that
    $$|\Th(z)|\le 1-M(1-|B_\Th(z)|)< 1-M(1-D), \quad z\in \D.$$
Thus the second inclusion is proved and the assertion follows.
\end{proof}

Recall that the Carleson curve $\Gamma_\e \subset \overline{\D}$ associated with $0<\e<1$ and an inner function $\Th$
has the following properties \cite{Cohn1983,Cohn1986, Garnett1981}:
    \begin{itemize}
    \item[\rm(1)] There exists $\e_0=\e_0(\e)\in (0,\e)$ such that $\e_0<|\Th(z)|<\e$ for $z\in \Gamma_\e\cap \D$.
    \item[\rm(2)] $\Gamma_\e\cap \D$ is a countable union of arcs $I_n$ with pairwise disjoint interiors such that
    \begin{itemize}
    \item[\rm $\bullet$] each $I_n$ is either a radial segment or part of a circle;
    \item[\rm $\bullet$] the end points $a_n$ and $b_n$ of $I_n$ satisfy
        $$\d_1\le \left|\frac{a_n-b_n}{1-\overline{a}_nb_n}\right|\le\d_2$$
    for all $n$ and some fixed $\d_1,\d_2\in (0,1)$.
    \end{itemize}
    \item[\rm(3)] If $\{z_n\}$ is the sequence of the middle points of $I_n$, then the Blaschke product $B_\Th$ with zeros $\{z_n\}$
    is an approximating Blaschke product of $\Th$.
    \end{itemize}
Now we are ready to prove Theorem~\ref{thm3}.

\medskip

\noindent
\emph{Proof of Theorem~\ref{thm3}.}
By Proposition~\ref{prop1}, the equivalence $\rm (i)\Leftrightarrow (iv)$ is valid.
Assume without loss of generality that $\{z_n(a)\}$ is ordered by increasing moduli, and
enumerate it such that, for all $k$, $r_j\le |\zjk(a)|<r_{j+1}$, $j=0,1,\ldots$, and $\{z_{jk}(a)\}$ is ordered by increasing moduli with $k$.
Then the hypothesis $\om \in \DD$ yields
    \begin{equation*}
    \begin{split}
    \sum_{j} \frac{\whw(r_j)\upsilon_j(a)}{(1-r_j)^{p-1}}\asymp \sum_j \sum_k \frac{\whw(z_{jk}(a))}{(1-|z_{jk}(a)|)^{p-1}}
    = \sum_n \frac{\whw(z_n(a))}{(1-|z_n(a)|)^{p-1}}, \quad a\in \D,
    \end{split}
    \end{equation*}
where $\upsilon_j(a)$ is as in Theorem~\ref{Theo1b}.
Consequently, $\rm (i)\Leftrightarrow (ii) \Leftrightarrow (iii)$ by Theorem~\ref{Theo1}.
Hence it suffices to prove $\rm (i)\Leftrightarrow (v)$.

By the hypotheses of $\om$, we know that $\whw(r)/(1-r)^p$ is essentially increasing with $r$.
Using this fact and condition (2) of $\Gamma_\e$, we obtain
    \begin{equation*}
    \begin{split}
    \int_{\Gamma_\e} \frac{\whw(z)}{(1-|z|)^p}\,|dz|=\sum_n \int_{I_n} \frac{\whw(z)}{(1-|z|)^p}\,|dz|
    \lesssim \sum_n |I_n| \frac{\whw(\xi_n)}{(1-|\xi_n|)^p},
    \end{split}
    \end{equation*}
where $\xi_n$ is the supremum of $I_n$ in the sense of absolute value.
Since, for all $n$,
    $$|I_n|\asymp |a_n-b_n| \asymp 1-|z_n| \quad \text{and} \quad \whw(\xi_n) \asymp \whw(z_n)$$
by condition (2) and the hypothesis $\om \in \DD$, we obtain
    $$\int_{\Gamma_\e} \frac{\whw(z)}{(1-|z|)^p}\,|dz| \lesssim \sum_n \frac{\whw(z_n)}{(1-|z_n|)^{p-1}}.$$
In a similar manner, one can also verify the asymptotic equation $\gtrsim$. Consequently,
condition (3) together with Lemmas~\ref{lemma1}~and~\ref{lemma2} yields
    \begin{equation*}
    \begin{split}
    \int_{\Gamma_\e} \frac{\whw(z)}{(1-|z|)^p}\,|dz|\asymp  \sum_n \frac{\whw(z_n)}{(1-|z_n|)^{p-1}}
    \asymp \|B_\Th'\|_{A_\om^p}^p \asymp \|\Th'\|_{A_\om^p}^p.
    \end{split}
    \end{equation*}
This means that $\rm (i)\Leftrightarrow (v)$ and the proof is complete. \hfill$\Box$

\section{Proof of Corollary~\ref{coro:Hp}, example and remarks}\label{Sec7}

Let us begin with the proof of Corollary~\ref{coro:Hp}.

\medskip

\noindent
\emph{Proof of Corollary~\ref{coro:Hp}.} Let $B_\Th$ be an approximation Blaschke product of $\Th$ with zeros $\{z_n\}$. Using this fact together with \cite[Theorem~2]{PR2016}, \cite[Theorem~1]{Cohn1983},
Lemmas~\ref{lemma1}~and~\ref{lemma2}, we obtain
    \begin{equation*}
    \begin{split}
    \|\Th'\|_{H^p}^p&\asymp \sup_{0\le r<1} \int_0^{2\pi} \left(\frac{1-|\Th(re^{i\t})|}{1-r}\right)^p d\t
    \asymp \sup_{0\le r<1} \int_0^{2\pi} \left(\frac{1-|B_\Th(re^{i\t})|}{1-r}\right)^p d\t \\
    &\asymp \|B_\Th'\|_{H^p}^p\asymp \sum_n (1-|z_n|)^{1-p} \asymp \|B_\Th'\|_{A_\a^{p+\a+1}}^{p+\a+1} \asymp \|\Th'\|_{A_\a^{p+\a+1}}^{p+\a+1}
    \end{split}
    \end{equation*}
for every/some $-1<\a<\infty$. Hence the equivalences $\rm (i)\Leftrightarrow (ii) \Leftrightarrow (iii)$ are valid.
Moreover, Theorem~\ref{thm3} gives $\rm (iii)\Leftrightarrow (iv) \Leftrightarrow (v) \Leftrightarrow (vi) \Leftrightarrow (vii)$.
This completes the proof.

\medskip

Next we give a concrete example in which we use Theorem~\ref{thm3}.

\begin{example}
Let us consider the atomic singular inner function
    $$S(z)=\exp\left(\frac{z+1}{z-1}\right), \quad z\in \D.$$
Let $\{z_n(a)\}$ be the zero-sequence of the Frostman shift $S_a$ of $S$, assume $a\in \D\setminus \{0\}$ and set $-\pi<\arg a\le \pi$. Solving the equation
    $$S(z)=\exp\left(\frac{z+1}{z-1}\right)=a,$$
we can present zeros $z_n(a)$ in the form
    $$z_n(a)=\frac{c_n+1}{c_n-1}, \quad \text{where} \quad c_n=\log |a|+i(2\pi n + \arg a),$$
for $n\in \Z$. It follows that
    \begin{equation*}
    \begin{split}
    1-|z_n(a)|^2&=\frac{|c_n-1|^2-|c_n+1|^2}{|c_n-1|^2}= \frac{-4\Real c_n}{|c_n-1|^2} \\
    &=\frac{-4\log |a|}{|c_n-1|^2} \asymp |n|^{-2}, \quad |n| \rightarrow \infty.
    \end{split}
    \end{equation*}
In particular, for $\a \in \R$,
    \begin{equation}\label{Eq:S-zeros}
    \begin{split}
    \sum_{n} (1-|z_n(a)|)^\a<\infty \quad \text{if and only if} \quad \a>\frac12.
    \end{split}
    \end{equation}
Hence, as a consequence of Theorem~\ref{thm3} and the nesting property $A_{\a_1}^p \subset A_{\a_2}^p$ for $-1<\a_1\le \a_2<\infty$, we obtain the following result: For $\frac12<p<\infty$ and $-1<\a<\infty$,
the derivative of $S$ belongs to $A_\a^p$ if and only if $\a>p-\frac32$.
This result originates to \cite{Jevtic1979}; see also \cite{MP1982}.
However, the argument used here is essentially different from that used in these references.

By \cite[Example~2]{GI1991}, the Frostman shift $S_a$ for any $a\in \D\setminus \{0\}$ is a Blaschke product with uniformly separated zeros. Applying this fact together with \eqref{Eq:S-zeros}, the above-mentioned result follows also from Lemma~\ref{lemma1}.
\end{example}

We close this note with the following remarks, which indicate two open questions.

\begin{itemize}
    \item[\rm(I)]
A modification of Corollary~\ref{coroBesov} for $p\ge q$ can be obtained in a similar manner as the current version
using Theorem~\ref{Theo1} instead of Theorem~\ref{Theo1b}. More precisely, the counterpart of \eqref{Eq:lemmaC}
takes the form
    \begin{equation}\label{Eq:lemmaC2}
    \sum_n (1-r_n)^{q/p-\a q}\upsilon_n(a)^{q/p}<\infty,
    \end{equation}
where $a\in \D\setminus E_\Th$. In addition, if $\Th$ belongs to $B_\a^{p,q}$ with the given restrictions, then there exists a set $\CZ \subset\D$ of logarithmic capacity zero
such that \eqref{Eq:lemmaC2} holds for every $a\in \D\setminus \CZ$.

Applying the above-mentioned result together with Corollary~\ref{coro:Hp}, one can show that, for $\frac12<p<\infty$, the derivative of an inner function $\Th$
belongs to $H^p$ if and only if $\Th''\in A_{p-1}^p$. Note that for $p\ge 1$ we are working with finite Blaschke products.
Originally this result was stated as a part of \cite[Theorem~3.10]{GGJ2011}.
The existence of the corresponding result in the case $0<p\le \frac12$ is an open question. However, since
    $$\{f:f'\in A_{p-1}^p\}\subset H^p, \quad 0<p\le 2,$$
by \cite[Lemma~1.4]{Vinogradov}, another implication is trivially valid also for $0<p\le \frac12$.
    \item[\rm(II)]
Corollary~\ref{coro:Hp} contains several ways to characterize those inner functions $\Th$
whose derivative belongs to $H^p$ for some $\frac12<p<1$. Nevertheless, it does not contain
an important characterization given in \cite[Theorem~1]{Grohn2016}:
For $\frac12<p<1$ and $1<\eta<\infty$, the derivative of an inner function $\Th$ belongs to $H^p$
if and only if $\Th$ is a Blaschke product whose zero-sequence $\{z_n\}$ satisfies the condition
        $$\int_0^{2\pi} \left(\sum_{z_n\in \Gamma_{\eta}(e^{i\t})} \frac{1}{1-|z_n|}\right)^p d\t<\infty,$$
where
    $$
    \Gamma_\eta(e^{i\t})=\{z\in\D:|z-e^{i\t}|\le\eta(1-|z|)\}.
    $$

For instance using Corollary~\ref{coro:Hp}, we may replace $H^p$ in the above-mentioned result by $A_\a^{p+\a+1}$, where $-1<\a<\infty$.
Even so any corresponding result for general $A_\a^p$ has not been verified, and proving such result seems to be laborious.
A reason for this is the fact that the argument of \cite[Theorem~1]{Grohn2016} utilizes the well-known identity
    $$\|B'\|_{H^p}^p= \frac{1}{2\pi}\int_0^{2\pi} \left(\sum_n \frac{1-|z_n|^2}{|z_n-e^{i\t}|^2}\right)^p\,d\t, \quad 0<p<\infty,$$
where $B$ is the Blaschke product with zeros $\{z_n\}$ \cite{AhernClark1974}; and we do not have a similar result for $A_\a^p$.
\end{itemize}

\medskip

\noindent
\textbf{Acknowledgements.}
The authors thank Janne Gr\"ohn for valuable comments, and the referees for careful reading of the manuscript.

\end{document}